\documentclass[reqno]{amsart}
\usepackage{amssymb,amsmath,hyperref}
\usepackage{amsrefs, float}
\usepackage[foot]{amsaddr}
\usepackage{bbold,stackrel, multicol}
 
\newcommand{\Z}{{\textsf{\textup{Z}}}}

\newtheorem{thm}{Theorem}

\newtheorem{cor}[thm]{Corollary}
\newtheorem{defi}[thm]{Definition}

\newtheorem{nota}[thm]{Notation}

\newtheorem{princ}[thm]{Principle}

\newtheorem*{tempo*}{Template}
\newtheorem{theorem}[thm]{Theorem}

\newtheorem{lemma}[thm]{Lemma}
\newtheorem{definition}[thm]{Definition}
\newtheorem{corollary}[thm]{Corollary}
\newtheorem{remark}[thm]{Remark}

\newcommand\be{\begin{equation}}
\newcommand\ee{\end{equation}}

\usepackage{amsmath,amsfonts}

\usepackage[applemac]{inputenc}

\def\bdefi{\begin{defi}\rm}
\def\edefi{\end{defi}}
\def\bnota{\begin{nota}\rm}
\def\enota{\end{nota}}
\def\FIVE{\Pi_{1}^{1}\text{-\textup{\textsf{CA}}}_{0}}
\def\SIX{\Pi_{2}^{1}\text{-\textsf{\textup{CA}}}_{0}}

\def\SIXK{\Pi_{k}^{1}\text{-\textsf{\textup{CA}}}_{0}^{\omega}}

\def\ZF{\textup{\textsf{ZF}}}
\def\ZFC{\textup{\textsf{ZFC}}}

\def\RCA{\textup{\textsf{RCA}}}
\def\({\textup{(}}
\def\){\textup{)}}

\def\RCAo{\textup{\textsf{RCA}}_{0}^{\omega}}
\def\ACAo{\textup{\textsf{ACA}}_{0}^{\omega}}

\def\N{{\mathbb  N}}

\def\R{{\mathbb  R}}

\def\SS{\textup{\textsf{S}}}

\def\di{\rightarrow}

\def\asa{\leftrightarrow}

\def\ACA{\textup{\textsf{ACA}}}

\def\QFAC{\textup{\textsf{QF-AC}}}
\def\QFDC{\textup{\textsf{QF-DC}}}

\def\DC{\textup{\textsf{DC}}}

\def\NIN{\textup{\textsf{NIN}}}

\def\TP{\textup{\textsf{Tp}}}
\def\Tp{\textup{\textsf{Tp}}}

\def\HBU{\textup{\textsf{HBU}}}
\def\HBT{\textup{\textsf{HBT}}}

\def\eps{\varepsilon}

\usepackage{mathrsfs}

\usepackage{graphicx}
\usepackage{tikz}
\usetikzlibrary{matrix, shapes.misc}
\usepackage{comment,tikz-cd}

\numberwithin{equation}{section}
\numberwithin{thm}{section}

\begin{document}
\title{The uncountability of the reals and the Axiom of Choice}
\author{Dag Normann}
\address{Department of Mathematics, The University 
of Oslo, P.O. Box 1053, Blindern N-0316 Oslo, Norway}
\email{dnormann@math.uio.no}
\author{Sam Sanders}
\address{Department of Philosophy II, RUB Bochum, Germany}
\email{sasander@me.com}

\begin{abstract}
The uncountability of the reals was first established by Cantor in what was later heralded as \emph{the first paper on set theory}.  
Since the latter constitutes the official foundations of mathematics, the logical study of the uncountability of the reals is a worthy endeavour for historical, foundational, and conceptual reasons.  
In this paper, we shall study the following principle:
\begin{center}
$\NIN_{[0,1]}$: \emph{there is no injection from the unit interval to the natural numbers.}
\end{center}
We show that relatively strong logical systems cannot prove $\NIN_{[0,1]}$.  In particular, the former system implies second-order arithmetic and fragments of the Axiom of Choice, including dependent choice.  
We also study the latter choice fragments in Kohlenbach's higher-order Reverse Mathematics.  
\end{abstract}
%
\maketitle              
%
\section{Introduction}
\subsection{Aim and motivation}\label{aimo}
The fact that infinity comes in different `sizes' was established by Cantor in the first paper on set theory (\cite{cantor1}), in the form of the uncountability of $\R$.
Given that set theory provides the current foundations of mathematics, the study of the uncountability of $\R$ is interesting for historical, foundational, and conceptual reasons.  
Below, we study the uncountability of $\R$ formulated as follows in Kohlenbach's higher-order Reverse Mathematics (\cite{kohlenbach2}). 
\begin{princ}[$\NIN_{[0,1]}$]
For any $Y:[0,1]\di \N$, there are $x,y\in [0,1]$ such that $x\ne_{\R}y$ and $Y(x)=_{\N}Y(y)$.  
\end{princ}
The authors have shown in \cite{dagsamX} that a relatively strong axiom system, namely $\Z_{2}^{\omega}+\QFAC^{0,1}$, cannot prove $\NIN_{[0,1]}$.  
This system is introduced in Section \ref{prelim} and implies second-order arithmetic and countable choice.
The goal of this paper is to extend this negative result, namely to show that $\NIN_{[0,1]}$ is not provable in $\Z_{2}^{\omega}+\bigcup_{\sigma}\QFAC^{\sigma, 1}$ where the latter fragment of the Axiom of Choice is as follows.  
\begin{princ}[$\QFAC^{\sigma, \tau}$] For finite types $\sigma, \tau$ and quantifier-free $\varphi$:  
\[
(\forall f^{\sigma})(\exists g^{\tau})\varphi(f, g)\di (\exists \Phi^{\sigma\di \tau})(\forall f^{\sigma})\varphi(f, \Phi(f)).
\]
\end{princ}
\noindent
Here, $\QFAC^{1,1}$ already gives rise to fragments of the Axiom of dependent Choice.  

\smallskip

Finally, our negative result is established in Section \ref{maink}.  
We also study $\QFAC^{\sigma, 1}$ and its kin in higher-order Reverse Mathematics in Section \ref{torralf}.  

\subsection{Preliminaries}\label{prelim}
We introduce some required definitions for the below.  We assume familiarity with Kohlenbach's higher-order Reverse Mathematics (abbreviated RM in the below), the base theory $\RCAo$ in particular.  
The original reference is \cite{kohlenbach2} with a recent `basic' introduction in \cite{sammetric}.  
We note that real numbers are defined in $\RCAo$ in the same way as in second-order RM, i.e.\ as fast-converging Cauchy sequences.  We stress that some of the below functionals were already 
studied by Hilbert and Bernays in the \emph{Grundlagen} (\cite{hillebilly2}).  
 
 \smallskip

First of all, we consider the following axiom where the functional $E$ is also called \emph{Kleene's quantifier $\exists^{2}$} and is discontinuous on Cantor space.  
\be\tag{$\exists^{2}$}
(\exists E:\N^{\N}\di \{0,1\})(\forall f\in\N^{\N})( E(f)=0 \asa (\exists n\in \N)(f(n)=0)).
\ee
We write $\ACAo\equiv \RCAo+(\exists^{2})$ and observe that the latter proves the same second-order sentences as $\ACA_{0}$ (see \cite{hunterphd}).
We shall mostly work in $\ACAo$, which is convenient as a set of reals $X\subset \R$ is defined via $F_{X}:\R\di \{0,1\}$ where $x\in X\asa F_{X}(x)=1$ for all $x\in \R$. 
Over $\RCAo$, $(\exists^{2})$ is equivalent to $(\mu^{2})$ (\cite{kohlenbach2}) where the later expresses that there is $\mu : \N^{\N}\di \N$ such that for $f\in \N^{\N}$ we have
\be\label{muf}
(\exists n\in \N)(f(n)=0)\di f(\mu(f))=0.
\ee
Secondly, consider the following axiom where the functional $\SS^{2}$ is often called \emph{the Suslin functional} (\cite{kohlenbach2, avi2, yamayamaharehare}):
\be\tag{$\SS^{2}$}
(\exists\SS:\N^{\N}\di \{0,1\})(\forall f \in \N^{\N})\big[  \SS(f)=0\asa ( \exists g \in \N^{\N})(\forall n \in \N)(f(\overline{g}n)=0) \big].
\ee
By definition, the Suslin functional $\SS^{2}$ can decide whether a $\Sigma_{1}^{1}$-formula in normal form, i.e.\ as in the right-hand side of $(\SS^{2})$, is true or false.   
The system $\FIVE^{\omega}\equiv \RCAo+(\SS^{2})$ proves the same $\Pi_{3}^{1}$-sentences as $\FIVE$ (see \cite{yamayamaharehare}).

\smallskip

Thirdly, we define the functional $\SS_{k}^{2}$ which decides the truth or falsity of $\Sigma_{k}^{1}$-formulas in normal form; we also define 
the system $\SIXK$ as $\RCAo+(\SS_{k}^{2})$, where  $(\SS_{k}^{2})$ expresses that $\SS_{k}^{2}$ exists.  
We define $\Z_{2}^{\omega}$ as $\cup_{k}\SIXK$ as one possible higher-order version of $\Z_{2}$.
The functionals $\nu_{n}$ from \cite{boekskeopendoen}*{p.\ 129} are just $\SS_{n}^{2}$ strengthened to return a witness (if existent) to the $\Sigma_{n}^{1}$-formula at hand.  
The operator $\nu_{n}$ is essentially Hilbert-Bernays' operator $\nu$ from \cite{hillebilly2}*{p.\ 479} restricted to $\Sigma_{n}^{1}$-formulas.

\smallskip

\noindent
Fourth, we introduce \emph{Kleene's quantifier $\exists^{3}$} as follows:
\be\tag{$\exists^{3}$}
(\exists E )(\forall Y:\N^{\N} \di \N)\big(E(Y)=0\asa\big[  (\exists f \in \N^{\N})(Y(f)=0)  \big]\big).
\ee
Both $\Z_{2}^{\Omega}\equiv \RCAo+(\exists^{3})$ and $\Z_{2}^\omega\equiv \cup_{k}\SIXK$ are conservative over $\Z_{2}$ (see \cite{hunterphd}).  
The functional $E$ from $(\exists^{3})$ is also called `$\exists^{3}$', and we use the same convention for other functionals.  Hilbert-Bernays' operator $\nu$ (see \cite{hillebilly2}*{p.\ 479}) is essentially Kleene's $\exists^{3}$, modulo a non-trivial fragment of the Axiom of (quantifier-free) Choice.

\section{A more negative result}\label{maink}
We establish the negative result regarding the uncountability of $\R$ sketched in Section \ref{aimo}.  
In particular, we show that $\Z_{2}^{\omega}+ \bigcup_{\sigma}~\QFAC^{\sigma, 1}$ does not prove $\NIN_{[0,1]}$.
To this end, we recall the model $\bf P$ introduced in \cite{dagsamX} in Section \ref{podel}.  We show in Section \ref{podel2} that ${\bf P}$ satisfies the aforementioned instances of the Axiom of Choice.

\subsection{A model of finite type arithmetic}\label{podel}
We introduce the model ${\bf P}$ from \cite{dagsamX} as it will be essential to our main result in Section \ref{podel2}.  
We first make our notion of `computability' precise as follows.  
\begin{enumerate}
\item[(I)] We adopt $\ZFC$, i.e.\ Zermelo-Fraenkel set theory with the Axiom of Choice, as the official metatheory for all results, unless explicitly stated otherwise.
\item[(II)] We adopt Kleene's notion of \emph{higher-order computation} as given by his nine clauses S1-S9 (see \cite{longmann}*{Ch.\ 5} or \cite{kleeneS1S9}) as our official notion of `computable'.
\end{enumerate}
For those familiar with \emph{Turing} computability, Kleene's S1-S9 in a nutshell is as follows: the schemes S1-S8 merely introduce (higher-order) primitive recursion, while S9 essentially states that the recursion theorem holds.

\smallskip

We refer to \cite{longmann} for a thorough and recent overview of higher-order computability theory.  
We do sketch one of our main techniques, called \emph{Gandy selection}.  
Intuitively speaking, this method expresses that S1-S9 computability satisfies an effective version of the Axiom of Choice.

\smallskip

Secondly, we have the following definition of `model'.  
\begin{definition}{\em A \emph{type structure} ${ \Tp}$ is a sequence $(\Tp[k])_{k \in \N}$ as follows.
\begin{itemize} 
\item $\Tp[0] = \N$.
\item For all $k \in \N$, $\Tp[k+1]$ is a set of functions $\Phi:\Tp[k] \rightarrow \N$.
\end{itemize}}
\end{definition} 
We note that $\Tp$ involves only total objects.  
The Kleene schemes can be interpreted for all type structures, by simply relativizing the definition. One of our applications of  type structures is that they serve as models for  fragments of higher-order arithmetic, structures over the \emph{language of finite types}. While the Kleene schemes are defined for pure types, the language of {finite} types is richer. However, assuming some modest closure properties of a type structure $ \Tp$, the extension to the finite types is unique (see \cite{longmann}*{\S4.2}). This is the case when $\Tp$ is \emph{Kleene closed} as in Definition \ref{uneek} right below.

\smallskip

\begin{definition}[Kleene computability]\label{uneek}
{\em  ~ \begin{itemize}
\item Let $\Tp$ be a type structure, let $\phi:\Tp[k] \rightarrow \N$, and let $\vec \Phi$ be in $\TP$.  We say that $\phi$ is \emph{Kleene computable} in $\vec \Phi$ (over $\TP$) if there is an index $e$ such that for all $\xi \in \Tp[k]$ we have that $\{e\}(\xi , \vec \Phi) = \phi(\xi)$.
\item  The type structure $\TP$ is \emph{Kleene closed} if for all $k$ and all $\phi:\Tp[k] \rightarrow \N$ that are Kleene computable in elements in $\TP$, we have that $\phi \in \Tp[k+1]$. 
\end{itemize}}
\end{definition}
When a type structure $\TP$ is Kleene closed, it will have a canonical extension to an interpretation $\Tp[\sigma]$ for all finite types $\sigma$.  
This is folklore and is discussed at length in \cite{longmann}*{\S 4.2}. 
We use $\TP^*$ to denote this unique extension. 
What is important to us is that if $\TP$ is Kleene closed, then $\TP^*$ is a model of $\RCA_0^\omega$ and all terms in G\"odel's $T$ have canonical interpretations in $\TP^*$.

\smallskip

Thirdly, the following (folklore) theorem shows that we have a high degree of flexibility when defining type structures from sets of functionals. 
\begin{theorem}\label{thm.ext} 
Let $A \subseteq \N^\N$ and let $B$ be a set of functionals $F:A \rightarrow \N$. 
Assume that all $f$ computable in a sequence from $B$ and $A$  are in $A$. 
Then there is a Kleene closed type structure $\TP$ such that $A = \Tp[1]$ and $B \subseteq \Tp[2]$.
\end{theorem}
Fourth, we introduce a version of \emph{Gandy selection}, first proved in \cite{supergandy}. 
Intuitively, $\lambda G.\{d\}(F, G)$ as in Theorem \ref{gandyselection} is a (partial) choice function with the biggest possible domain.
A functional $F^{2}$ is \emph{normal} if it computes $\exists^{2}$.  
\begin{theorem}[Gandy Selection]\label{gandyselection}
Let $F^{2}$ be normal. Let $A \subset \N \times \N^{\N^\N}$ and $e$ be such that $(a,G^2) \in A$ if and only if $\{e\}(F,G,a)$ terminates \($A$ is semi-computable in $F$\). Then there is an index $d$ such that $\{d\}(F,G)$ terminates if and only if there exists $a\in\N$ such that $(a,G) \in A$, and then $\{d\}(F,G)$ is one of these numbers.
\end{theorem}
\begin{proof}
See \cite{longmann}*{\S5.4} or \cite{Sacks.high}*{p.\ 245}.
\end{proof}
\begin{remark}\label{remark1}{\em We need this strong version of Gandy selection in the proof of Lemma~\ref{3}. 
When $F$ is normal, computing relative to $F$ and $G^2$ satisfies \emph{stage comparison}, a soft requirement for Gandy selection. 
Replacing $G$ with an arbitrary functional of type $\geq 3$, stage comparison will not be available anymore, and thus neither Gandy selection. 
We can of course replace $\N^{\N^\N}$ by $\N^\N$ in Theorem \ref{gandyselection}, as we do in the proof of Corollary \ref{cor.inj}.
}\end{remark}
Several  of our applications of Gandy selection are based on the following corollary. 
Intuitively speaking, the functional $G$ computes an $F$-index for $f$.
\begin{corollary}\label{cor.inj} 
Let $F^{2}$ be normal. Then there is a partial functional $G$ computable in $F$ which terminates if and only if the input $f\in \N^{\N}$ is computable in $F$, and such that we have $G(f) = e\di (\forall a\in \N)(f(a) = \{e\}(F,a))$.  
\end{corollary}
\begin{proof} When $F$ is normal, the relation $(\forall a \in \N)(f(a) = \{e\}(F,a))$ is clearly\\ semi-computable, and we can apply Gandy Selection.\end{proof}
Note that the functional $G$ is always injective.
Of course, these results are equally valid for all Kleene closed type structures, and we may replace $\N^{\N^\N}$ in Theorem \ref{gandyselection} by any finite product of pure types $\leq 2$.

\smallskip

Next, we define the Kleene closed type structure ${\bf P}$ which is crucial for the below independence results involving $\Z_{2}^{\omega}$.
We note that ${\bf P}$ is constructed under the set-theoretical assumption that $\textsf{V = L}$.  
There is no harm in this, since what is of interest is the logic of the structure, which statements are true and which are false, and our results will not depend on the assumption that $\textsf{V = L}$; they are proved within $\ZF$.
Now recall the functionals $\SS^{2}_{k}$ from Section \ref{prelim}. We use the assumption $\textsf{V = L}$ motivated by the following fact from set theory.
\begin{lemma}[$\textsf{V = L}$] Let $A \subseteq \N^\N$ be closed under computability relative to all $\SS^{2}_{k}$. Then all $\Pi^1_n$-formulas are absolute for $A$ for all $n$. \end{lemma}
\begin{proof} 
For $n \leq 2$, this is a general fact independent of the assumption \textsf{V = L}, and for $n > 2$ it  is a consequence of the existence of a $\Delta^1_2$-well-ordering of $\N^\N$. 
\end{proof}
\begin{definition}[$\textsf{V = L}$]\label{defP}{\em 
Let $\SS^{2}_{\omega}$ be the join of all $\SS^{2}_{k}$, and let ${\bf P}$ be the Kleene closed  type-structure, as obtained from Theorem \ref{thm.ext}, where ${\bf P}[1]$ is the set of functions computable in $\SS^{2}_{\omega}$ and the restriction of $\SS^{2}_{\omega}$ to ${\bf P}[1]$ is in ${\bf P}[2]$.}
\end{definition}
The model ${\bf P}$, under another name, has been used to prove \cite{dagsamV}*{Theorem 4.3}.
Recall the unique extension $\TP^{*}$ of $\TP$ introduced below Definition \ref{uneek}.
The principle $\NIN_{2^{\N}}$ expresses that there is no injection from Cantor space to the naturals.  
\begin{thm}\label{lemmaP}
The type structure  ${\bf P}^*$ derived  from ${\bf P}$ as defined above is a model for $\Z_{2}^{\omega} + \QFAC^{0,1}+\neg\NIN_{2^{\N}}$. 
 \end{thm}
\begin{proof} 
We assume that $\textsf{V = L}$, which implies that all $\Pi^1_n$-formulas are absolute for ${\bf P}[1]$.  Since ${\bf P}[1]$ is closed under computability relative to each $\SS^{2}_k$, we have that ${\bf P}[1]$ satisfies all $\Pi^1_n$-comprehension axioms. Now assume that $(\forall n^{0}) (\exists f^{1}) Q(n,f,\vec \Phi)$  is true in ${\bf P}$, where $Q$ is quantifier-free and $\vec \Phi$ is a list of parameters from ${\bf P}$.  
Since all functionals in $\vec{\Phi}$  are computable in $\SS^2_\omega$, the set 
\[
S := \{(n,f)\in \N\times\N^{\N} : Q(n,f,\vec{\Phi})\} 
\]
is computable in $\SS^2_\omega$. Here, we use the substitution theorem for S1-S9 from \cite{kleeneS1S9} relativised to the model $\bf P$.
Further, the set \textsf{TOT} of indices $e^{0}$ for total functions $f\in \N^{\N}$ relative to any total functional $F^2$ is semi-computable in $F$ since
\[
e \in \textsf{\textup{TOT}} \asa  \big[F\big(\lambda a^0.\{e\}(a,F) \big)\textup{ terminates}\big].
\]
As a consequence, and technically by a second use of the aforementioned substitution theorem, the following set is semi-computable in $\SS^{2}_{\omega}$:
 \[
 R:=\{(n, e)\in \N^{2} : (\exists f \in \N^{\N})[ Q(n,f,\vec \Phi)\wedge (\forall a \in \N)(f(a) = \{e\}(\SS_\omega^{2},a))] \}.
 \]
 Moreover, we have that  $(\forall n^{0}) (\exists e^{0})[(n,e)  \in R]$.
 By assumption and Gandy selection, there is a function $g$ computable in $\SS^{2}_{\omega}$ such that $R(n,g(n))$ for all $n^{0}$. 
 If $G(n)$ is the function $f$ computed from $\SS^{2}_{\omega}$ with index $g(n)$, we have that $G^{0\di 1} \in \bf P$ witnesses this instance of quantifier-free choice.
 
 \smallskip
 
Regarding $\NIN_{2^{\N}}$, recall that $\mathbf{P}$ consists of all objects computable in $\SS_{\omega}^{2}$, the union of all $\SS_{k}^{2}$.  
Thus, for any $f\in 2^{\N}\cap \mathbf{P}$ there is some (unique minimal) $e_{f}\in \N$ such that the $e_{f}$-th Kleene algorithm computes $f$ in terms of $\SS_{\omega}^{2}$; by Gandy selection, the choice function $\lambda f^{1}.e_{f}$ is already part of $\mathbf{P}$ and provides an injection from $2^{\N}\cap \mathbf{P}$ to $\N$, i.e.\ $\NIN_{2^{\N}}$ is false in $\mathbf{P}$.  
More formally, ${\bf P}[2]$ contains a functional $G:A \rightarrow \N$ that is \emph{injective}, which is a direct consequence of Corollary~\ref{cor.inj} using $\SS^{2}_\omega$ for $F$.
 \end{proof} 
 Finally, it is straightforward to prove the equivalence $\NIN_{[0,1]}\asa \NIN_{2^{\N}}$, say over $\ACAo$ (see \cite{samcie22} for a proof).  
 In particular, the latter system provides a uniform mechanism for converting reals in $[0,1]$ to binary representation.

\subsection{Main result}\label{podel2}
We show that $\NIN_{[0,1]}$ cannot be proved in $\Z_{2}^{\omega}+\bigcup_{\sigma}\QFAC^{\sigma, 1}$ by showing that the model ${\bf P}$ from Section \ref{podel} satisfies the latter but not the former.

\smallskip

First of all, we establish two lemmas needed for our main result.
Recalling that the only predicate of our language is equality at base type $0$, the following claim is just an observation that does not require a proof.
\begin{lemma}\label{2} 
Let $\Delta$ be a quantifier-free formula with higher-order parameters $\vec \Phi$ and free variables $x^\tau$ and $y^1$. Then there is a functional $\Psi^{\tau\di 2}$, computable in $\vec \Phi$, such that for each $\phi^\tau$ and $f^1$ we have
\[
\Delta(\phi,f,\vec \Phi) \leftrightarrow \Psi(\phi)(f) = 0.
\]
\end{lemma}
Clearly, Lemma \ref{2} is also valid in $\bf P$, and the same is true for the next one.
\begin{lemma}\label{3} 
Let $F$ be a normal functional of type 2. Then there is a partial functional $\Theta$ of type $2 \rightarrow 1$ such that for all $G$ of type 2, the following are equivalent:
\begin{itemize}
\item {the function $\Theta(G)$ is total and $G(\Theta(G)) = 0$},
\item {there exists a $f^1$ computable in $F$ and $G$ such that $G(f) = 0$. }
\end{itemize}
\end{lemma}
\begin{proof}
This is proved via Gandy selection. For $G$ of type 2, define
\[
A_{G}:=\big\{e\in \N: \textup{$\lambda a^{0}.\{e\}(a,F,G)$ is total and $G$ maps the former to $0$}\big\}. 
\]
Clearly, $A_{G}$ is semi-computable in $F$ and $G$, uniformly in $G$.
By Gandy selection, there is a partial computable functional $\xi$ selecting $e_{0}\in A_{G}$ whenever the latter is nonempty. Composing $\xi$ with $e \mapsto \lambda a.\{e\}(a,F,G)$ yields the required $\Theta$.
\end{proof}
We may now prove the main theorem of this section. 
\begin{thm}\label{masterpiece}
The system $\Z_{2}^{\omega}+ \bigcup_{\sigma}~\QFAC^{\sigma, 1}$ does not prove $\NIN_{[0,1]}$. 
\end{thm}
\begin{proof}
We let ${\bf P}$ be the model from Definition \ref{defP}.  
By Theorem \ref{lemmaP}, ${\bf P}$ satisfies $\Z_{2}^{\omega}+\neg\NIN_{[0,1]}$.  
Now assume that ${\bf P} \models \forall x^\tau \exists y^1\Delta(x,y ,\vec \Phi)$ where $\Delta$ is quantifier-free. 
We work inside $\bf P$ and use the above lemmas.  Consider $\Psi$ provided by applying Lemma \ref{2} for $\Delta$; $\Psi$ is computable in $\SS^2_\omega$ since all parameters in $\Delta$ are  computable in $\SS^2_\omega$.  Similarly, apply Lemma \ref{3} for $F = \SS^2_\omega$ and let $\Theta$ be the resulting functional.
By assumption on $\Delta$, we have $(\forall x^{\tau})\Delta(x, \Theta(\Psi(x)), \vec \Phi)$.  Moreover, $\lambda x^{\tau}.\Theta(\Psi(x))$ is computable in $\SS_{\omega}^{2}$, and therefore already included in {\bf P}.  
\end{proof}
\begin{remark}\label{remark2}
{\em  We cannot replace $\QFAC^{\sigma, 1}$ by $\QFAC^{\sigma, \tau}$ for types $\tau$ of any higher order in the above argument for two reasons. One reason is discussed in Remark~\ref{remark1}. The other reason is that a normal functional of type 2, like $\SS_\omega^{2}$, can decide equality over $\N^\N$, but not for types at higher levels.
}\end{remark}
The theorem has some interesting implications for the following fragment of the Axiom of dependent Choice.  
\begin{princ}[$\textsf{QF}$-$\DC^{1,1}$] Let $\varphi$ be quantifier-free and such that $(\forall n\in \N)(\forall f\in \N^{\N})(\exists g\in \N^{\N})\varphi(n, f, g)$.  
Then there is $(f_{n})_{n\in \N}$ such that $(\forall n\in \N)\varphi(n, f_{n}, f_{n+1})$.  
\end{princ}
The axiom $\QFDC^{1,1}$ is similar in kind to Kohlenbach's fragment of countable choice $\QFAC^{0,1}$ from \cite{kohlenbach2}.
We now have the following immediate corollary.  
\begin{cor}\label{flegm}
The system $\Z_{2}^{\omega}+\QFDC^{1,1}$ cannot prove $\NIN_{[0,1]}$.
\end{cor}
\begin{proof}
We shall show that $\QFDC^{1,1}$ follows from $\QFAC^{1,1}$ and $({\bf R}_{1})$, where the latter formalises primitive recursion of type $1\di 1$ objects.
The recursor constant $\mathbf{R}_{\sigma}$ is defined as follows in general:
\be\label{specialty}\tag{${\bf R}_{\sigma}$}
\mathbf{R}_{\sigma}(f, g, 0):= g \textup{ and } \mathbf{R}_{\sigma}(f, g, n+1):= f(n, \mathbf{R}_{\sigma}(f, g, n)), 
\ee
where $g^{\sigma}$ and $f^{(0\times \sigma)\di \sigma}$ are arbitrary and $\sigma$ is any finite type. 
Since G\"odel's $T$ is included in S1-S9 computability, the corollary follows from the (proof of the) theorem. 
For fixed quantifier-free $\varphi$, apply $\QFAC^{1,1}$ to 
\[
(\forall n\in \N)(\forall f\in \N^{\N})(\exists g\in \N^{\N})\varphi(n, f, g).
\]
Let $\lambda n.\lambda f.\Phi(n, f)$ be the resulting functional and use $\textbf{R}_{1}$ to define $f_{0}:=00\dots$ and $f_{n+1}:= \Phi(n, f_{n})$, as required for $\QFDC^{1,1}$.  
\end{proof}
In conclusion, Theorem \ref{masterpiece} suggests that we may use $\cup_{\sigma}\QFAC^{\sigma, 1}$ and $\QFDC^{1,1}$ in the RM of $\NIN_{[0,1]}$ and beyond, which is the content of the following section.

\smallskip

\section{Some applications}\label{torralf}
We discuss how quantifier-free choice as in $\QFAC^{\sigma, 1}$ is useful in higher-order RM.  
For instance, showing that a set is uncountable can be done in various ways and converting between these seems to require fragments of $\QFAC^{\sigma, 1}$.

\smallskip

In more detail, we establish new equivalences for $\NIN_{[0,1]}$ using $\QFAC^{\sigma, 1}$ in Section \ref{klapper} and provide an overview of how $\QFAC^{\sigma, 1}$ has been (or can be) used in RM in Section \ref{fruitcake}.  
We stress that the negative results in Section~\ref{podel2} justify the use of $\RCAo+\QFAC^{\sigma, 1}$ as a base theory for the RM of $\NIN_{[0,1]}$ and stronger principles.
We also obtain an equivalence for $\QFDC^{1,1}$ in Theorem \ref{chuchu}.

\subsection{Quantifier-free choice in the base theory}\label{klapper}
We establish some equivalences for $\NIN_{[0,1]}$ using $\QFAC^{\sigma,1}$ in the base theory.
We shall consider metric and second-countable spaces, which have been previously studied in RM (\cites{simpson2, damurm, samSECOND, samHARD,sammetric}). 
The associated definitions in higher-order RM are the textbook ones (see \cites{samSECOND, samHARD, sammetric}), i.e.\ we shall not introduce the former as they are well-known.

\smallskip

First of all, the following result seems to crucially depend on $\QFAC^{1, 1}$; item \eqref{tajel1} is well-known from textbooks \(\cite{munkies}*{Ch.~3, Prob.~4} and \cite{rudin}*{Ch.\ 2, Ex.\ 19}).  
\begin{thm}[$\ACAo+\QFAC^{1,1}$]\label{fijally}
The following are equivalent. 
\begin{enumerate}
\renewcommand{\theenumi}{\alph{enumi}}
\item The uncountability of the reals as in $\NIN_{[0,1]}$.
\item For $X\subset \R$, if $f:X\di [0,1]$ is surjective, then $X$ is uncountable$^{\ref{pufferfish}}$.\label{tajel0}
\item Let $(M, d)$ be a metric space with $M\subset \R$ and let $A\subset M$ be connected and satisfy $|A|\geq 2$. \label{tajel1}
Then $A$ is  uncountable\footnote{Similar to $\NIN_{[0,1]}$, `$A$ is uncountable' for $A\subset M$ in a metric space $(M, d)$ means that there is no injection from $A$ to $\N$, i.e.\ for any $Y:M\di \N$, there are $x, y\in A$ with $x\ne_{M}y\wedge Y(x)=Y(y)$.  Note that $d(x,y)=_{\R}0\asa x=_{M}y$ by the definition of metric space.\label{pufferfish}}.
\item Let $(M, d)$ be a connected metric space with $M\subset \R$ and $f:M\di \R$ continuous.  \label{tajel2}
If $f(M)$ exists and has two distinct elements, it is uncountable.  
\item The previous item for $f:M \di M_{0}$ where $(M_{0}, d_{0})$ is any metric space. \label{tajel3}
\item Let $(M_{i}, d_{i})$ for $i=0,1$ be metric spaces with $M_{i}\subset \R$, $|M_{0}|\geq 2$, and $M_{0}$ connected, and let $f:M_{1}\di M_{0}$ be surjective. Then $M_{1}$ is uncountable. \label{tajel1337}
\item Any connected \emph{functionally Hausdorff}\footnote{A topological space $X$ is functionally Hausdorff if for any distinct $x, y\in X$, there is continuous $f:X\di [0,1]$ such that $f(x)=0$ and $f(y)=1$.  
} second-countable space over the reals with at least two points is uncountable.\label{tajel4}
\end{enumerate}
We do not need $\QFAC^{1,1}$ to prove item \eqref{tajel2}.
\end{thm}
\begin{proof}
That items \eqref{tajel0}-\eqref{tajel4} imply $\NIN_{[0,1]}$ is straightforward.  
For the implication $\NIN_{[0,1]}\di \eqref{tajel0}$, let $X\subset \R$ and $f:X\di [0,1]$ be as in the latter and apply $\QFAC^{1,1}$ to $(\forall r\in [0,1])(\exists x\in X)(f(x)=_{\R}r)$ to obtain $\Phi^{1\di 1}$.
In case $Y:\R\di \N$ is injective on $X$, $\lambda r. Y(f(r))$ is injective on $[0,1]$, as required.

\smallskip

For the implication $\NIN_{[0,1]}\di \eqref{tajel1}$, suppose the latter is false, i.e.\ let $A\subset M$ be countable and connected.  
Let $a, b\in A$ be distinct points and consider the following
\be\label{texxxo}
(\exists \eps \in (0, d(a, b)))(\forall x\in A)( d(a, x)\ne_{\R} \eps ).
\ee
Let $\eps$ be as in \eqref{texxxo} and consider the following open sets:
\[
U:=\{x\in A: d(a, x)<\eps\} \textup{ and } U:=\{x\in A: d(a, x)>\eps\}
\]
Since $U\cup V=A$ and $U, V\ne \emptyset$, this contradicts the connectedness of $A$, i.e.\ \eqref{texxxo} must be false.  
Now apply $\QFAC^{1, 1}$ to the negation of \eqref{texxxo}, as follows:
\be\label{texxxo2}
(\forall \eps \in (0, d(a, b)))(\exists x\in A)( d(a, x)=_{\R} \eps ),
\ee
and let $\Phi^{1\di  1}$ be the resulting choice function.  
Since $A$ is countable, there is $Y:M\di \N$ that is injective on $A$. 
Then $\lambda x. Y(\Phi(x))$ is an injection from the real interval $(0, d(a, b))$ to $\N$, which contradicts $\NIN_{[0,1]}$.

\smallskip

For the implication $\NIN_{[0,1]}\di \eqref{tajel2}$, we prove that for continuous $f:M\di \R$, $f(M)$ is connected if $M$ is.  
Indeed, let $M$ be connected and suppose $f(M)=U\cup V$ where $U, V\subset \R$ are open and non-empty.  
Now consider $X= f^{-1}(U)$ and $Y= f^{-1}(V)$, which are open as topological and `epsilon-delta' continuity are equivalent by definition in $\RCAo$.    
Since $M$ is connected, $M=X\cup Y$ implies that there is $z\in X\cap Y$, yielding that $f(z)\in U\cap V$, as required for the connectedness of $f(M)$.  
If $|f(M)|>2$, then there are $x_{0}, y_{0}\in M$ such that $f(x_{0})\ne_{\R} f(y_{0})$. To show that $(f(x_{0}), f(y_{0}))\subset f(M)$, suppose $w\in (f(x_{0}), f(y_{0}))$ and $w\not \in f(M)$.  
Then $U=\{x\in M:  f(x)>_{\R}w\}$ and $V=\{x\in M:f(x)<_{\R}w\}$ are open since $f$ is continuous, contradicting that $M$ is connected.  Hence, $f(M)$ contains an interval and is therefore uncountable by $\NIN_{[0,1]}$, as required.   
For item \eqref{tajel3}, use the same proof to show that $f(M)$ is connected and apply item \eqref{tajel1}.

\smallskip

To prove item \eqref{tajel4}, let $X$ be as in the latter, i.e.\ for any two distinct points $x, y\in X$ there is continuous $f:X\di [0,1]$ such that $f(x)=0$ and $f(y)=1$.  
Now if $x_{0}\in (0,1)$ is not in the range of $f$, then the open sets $U=f^{-1}([0,x_{0}))$ and $f^{-1}((x_{0}, 1])$ are non-empty, disjoint, and such that $X=U\cup V$, a contradiction.
Now apply $\QFAC^{1,1}$ to $(\forall r\in [0,1])(\exists x\in X)(f(x)=_{\R} r)$ to obtain $\Phi^{1\di 1}$ such that $x=\Phi(r)$ in the previous formula.  
Now, if $Y:X\di \N$ is injective, then so is $Z:[0,1]\di \N$ defined by $Z(r):= Y(\Phi(r))$, contradicting $\NIN_{[0,1]}$.

\smallskip

To prove item \eqref{tajel1337}, let $f:M_{1}\di M_{0}$ be as in the latter, i.e.\ $(\forall y\in M_{0})(\exists x\in M_{1})(f(x)=y)$; applying $\QFAC^{1,1}$ yields $\Phi^{1\di 1}$ such that $y=\Phi(x)$ in the latter.  
In case $(M_{1}, d)$ is countable, there is $Y:\R\di \N$ such that $Y(x)=Y(y)$ implies $x=_{M_{1}}y$ for any $x, y\in M_{1}$.  Now define $Z:\R\di \N$ as $\lambda x.Y(\Phi(x)) $, which satisfies $Z(x)=Z(y)\di x=_{M_{0}}y$ for all $x, y\in M_{0}$, essentially by definition.  Hence $(M_{0}, d)$ is countable, contradicting item \eqref{tajel1} as required.  
\end{proof}
We can replace the existential quantifier in \eqref{texxxo2} by `$(\exists! x\in A)$', where uniqueness is however expressed relative to `$=_{M}$'.   
We also obtain weaker results as follows.
\begin{cor}[$\ACAo$] 
Items \eqref{tajel0}-\eqref{tajel4} can be proved for `uncountable' replaced by `non-enumerable'.  
\end{cor}
\begin{proof}
We first consider item \eqref{tajel0} formulated with `$X$ is non-enumerable'.  
Now if $X$ were enumerable, say given by the sequence $(x_{n})_{n\in \N}$, the surjectivity of $f:X\di \R$ implies $(\forall r\in [0,1])(\exists n\in \N )(f(x_{n})=_{\R}r)$.
Hence, $(f(x_{n}))_{n\in \N}$ enumerates $[0,1]$, which even $\RCA_{0}$ disproves (see \cite{simpson2}*{II.4.9}).
A similar proof works for item \eqref{tajel1}: if $A$ is given by a sequence $(a_{n})_{n\in \N}$, then \eqref{texxxo2} implies 
\be\label{texxxo25}
(\forall \eps \in (0, d(a, b)))(\exists n\in \N)( d(a, a_{n})=_{\R} \eps ).
\ee
Hence, the interval $(0, d(a, b))$ can be enumerated via $(d(a, a_{n}))_{n\in \N}$, a contradiction as for item \eqref{tajel0}.
The other items are proved in the same way.   
\end{proof}
Next, we have the following theorem where we recall that $({\bf R}_{1})$ formalises primitive recursion of type $1\di 1$ objects (see Corollary \ref{flegm}).  
Item~\eqref{slafke} is studied in textbooks (\cite{munkies}*{p.\ 176}) and second-order RM via codes (\cite{simpson2}*{II.5.9}).
\begin{thm}[$\ACAo+\QFAC^{1,1}+({\bf R}_{1})$]\label{dstiff} The following are equivalent.  
\begin{enumerate}
\renewcommand{\theenumi}{\alph{enumi}}
\item The uncountability of the reals as in $\NIN_{[0,1]}$.
\item A perfect set of reals is not countable. \label{slafke0}
\item A complete metric space $(M, d)$ with $M\subset \R$ and with \emph{no isolated points}, is uncountable$^{\ref{pufferfish}}$.\label{slafke}
\item The previous item restricted to \emph{separable} spaces. \label{slafke2} 
\item Let $(M_{i}, d_{i})$ be complete metric spaces without isolated points and $M_{i}\subset \R$. For continuous $f: M_{0} \di M_{1}$, if the set $f(M_{0})$ exists, it is uncountable.\label{ucomo} 
\item The previous item restricted to sequential continuity. \label{zefferinthesky}
\end{enumerate}
We do not need $({\bf R}_{1})$ to establish item \eqref{slafke0} or \eqref{slafke2}.  
\end{thm}
\begin{proof}
That $\NIN_{[0,1]}$ follows from the other items is straightforward.    We now prove that the former implies item \eqref{slafke} (and item \eqref{slafke0}). 
To this end, we first observe two applications of $\QFAC^{1,1}$ for a metric space $(M, d)$ as in item \eqref{slafke}:
\begin{itemize}
\item  $(M, d)$ has no isolated points, i.e.\ $(\forall x\in M, k\in \N)(\exists y\in M)(0< d(x, y)< \frac{1}{2^{k}} )$.  Hence, there is $\Phi_{0}^{1\di 1}$ such that $y=_{1}\Phi_{0}(x, k)$ in the latter. 
\item  $(M, d)$ is complete, i.e.\ we have
\be\label{Komplete}\textstyle
(\forall (x_{n})_{n\in \N})(\exists x\in M)[  \textup{$(x_{n})_{n\in \N}$ is Cauchy } \di  \lim_{n\di \infty }x_{n}=_{M}  x ].
\ee
The formula in square brackets in \eqref{Komplete} is arithmetical, i.e.\ there is $\Phi_{1}^{1\di 1}$ such that $x=_{1}\Phi_{1}(\lambda n.x_{n})$ in \eqref{Komplete}.

\end{itemize}
One standard technique to show that $M$ is uncountable is then as follows:  one constructs an injection $f:2^{\N}\di M$ based on the below items.  
\begin{itemize}  
\item Use $\Phi_{0}$ and ${\bf R}_{1}$ to define a sequence of pairwise different points $(x_{\sigma})_{\sigma \in 2^{<\N}}$ in $M$ such that $x_{\sigma*i*j} \in B(x_{\sigma*i}, \frac{1}{2}d(x_{\sigma*0}, x_{\sigma*1}))$.  
\item For $\alpha \in 2^{\N}$, define $f(\alpha)=x_{\alpha}$ in case $\alpha $ has a tail of zeros.  Otherwise, use $\Phi_{1}$ to define $f(\alpha)$ as the limit of the sequence $(x_{\overline{\alpha}n})_{n\in \N}$ in $M$.  
\item Verify that $f:2^{\N}\di M$ is an injection `by definition'.  
\end{itemize}
Since the space $2^{\N}$ does not really involve representations (in contrast to $\R$), the issue of extensionality relative to $=_{M}$ does not pose problems.  
For item \eqref{slafke2} (and the final sentence of the theorem) observe that the sequence $(x_{\sigma})_{\sigma \in 2^{<\N}}$ is readily defined in terms of a dense subsequence of $M$, say over $\ACAo$.

\smallskip

Finally, let $(M_{i}, d_{i})$ and $f:M_{0}\di M_{1}$ be as in item \eqref{ucomo}.  By the continuity of $f$ and the completeness of $M_{0}$, $f(M_{0})$ is complete and cannot have isolated points.  
One seems to need $\QFAC^{0,1}$ for the latter completeness.  Hence, item \eqref{slafke} implies item~\eqref{ucomo}, using $\QFAC^{1,1}$.   
For item \eqref{zefferinthesky}, the proof using sequential continuity is essentially the same, again using $\QFAC^{0,1}$ in an essential way.  
\end{proof}
Next, we formulate an equivalence for $\QFDC^{1,1}$ based on \cites{bish1, heerlijkheid}, where we focus on what is exactly needed for equivalences to $\QFDC^{1,1}$.  It should be possible to obtain more equivalences based on the results in \cite{samBIG2, dagsamVII}.
The notion of R2-open set\footnote{A set $O\subset M$ is R2-open if there is $Y:M\di \R$ with $x\in O\asa [Y(x)>_{\R}0\wedge B(x, Y(x))\subset O]$.} in item \eqref{taxxx} of Theorem \ref{chuchu} was first introduced in \cite{dagsamVII}.  
\begin{thm}[$\ACAo$]\label{chuchu}
The following are equivalent. 
\begin{enumerate}
\renewcommand{\theenumi}{\alph{enumi}}
\item The principle $\QFDC^{1,1}$.
\item The Baire category theorem for metric spaces $(M, d)$ with $M\subset \N^{\N}$ and sequences of dense and {R2-open} sets. \label{taxxx} 
\item Item \eqref{taxxx} for $M$ and the graph of $d$ defined by arithmetical formulas.  
\end{enumerate}
\end{thm}
\begin{proof}
The standard (constructive) proof of the Baire category theorem is well-known (see \cite{bish1}*{p.\ 87}). 
Modulo coding, this is the proof used in second-order RM (\cite{simpson2}*{II.5.8}).  
The same proof goes through for complete metric spaces as in item~\eqref{taxxx} if we use $\QFDC^{1,1}$ to define the Cauchy sequence in the former proof.  
Thus, item~\eqref{taxxx} follows from $\QFDC^{1,1}$.

\smallskip

To show that item \eqref{taxxx} implies $\QFDC^{1,1}$, let $M$ be the set of all sequences in Baire space (readily coded as a subset of Baire space), i.e.\ $M$ consists of all objects of type $0\di 1$.  
The metric $d:M^{2}\di \R$ is then as follows for any $f, g\in M$:
\[
d(f, g)=
\begin{cases}
0 & \textup{if $(\forall n\in \N)(f(n)=_{1}g(n) )$} \\
\frac{1}{2^{(\mu n)(f(n)\ne_{1}g(n))}} & \textup{otherwise}
\end{cases}.
\]
That $(M, d)$ is complete follows by noting that a Cauchy sequence $(f_{n})_{n\in \N}$ in $M$ converges to $g\in M$ defined as $g(n):=f_{F(n)}(n)$ where $F^{1}$ is such that 
\[\textstyle
(\forall k\in \N)(\forall m, n \geq F(k))(d(f_{n}, f_{m})<\frac{1}{2^{k}}).
\]  
Now let $\varphi$ be quantifier-free and such that $(\forall x\in \N^{\N})(\exists y\in \N^{\N})\varphi(x, y)$ and define sets $O_{n}\subset M$ as follows, for any $f^{0\di 1}\in M$:
\[
f\in O_{n}\asa (\exists m\in \N)\varphi(f(n), f(m)).
\]  
Then each $O_{n}$ is open since $f\in O_{n}$ implies 
\be\label{taxxation}
(\exists k\in \N)(\forall g\in M)(\overline{f}k=_{1^{*}}\overline{g}k \di g\in O_{n} )
\ee
by definition, implying $B(f, \frac{1}{2^{k+2}})\subset O_{n}$ for $k\in \N$ as in \eqref{taxxation}.  
In this light, the functional $\lambda f.\lambda n.(\mu m)\varphi(f(n), f(m))$ readily provides the function $Y_{n}$ required for the R2-representation of $O_{n}$.  
That each $O_{n}$ is dense in $M$ follows by the assumption $(\forall x\in \N^{\N})(\exists y\in \N^{\N})\varphi(x, y)$.  
Hence, there is $h\in \cap_{n\in \N}O_{n}$ by item \eqref{taxxx}, i.e.\ $(\forall n\in \N)(\exists m\in \N)\varphi(h(n), h(m)) $, by definition.   
Using $\QFAC^{0,0}$, available in $\RCA_{0}$, we obtain the required sequence and $\QFDC^{1,1}$ follows.  
\end{proof}
By Corollary \ref{flegm}, the Baire category theorem for R2-open sets does not imply $\NIN_{[0,1]}$; 
by contrast, the Baire category theorem for open sets (without R2-representation) readily implies item \eqref{slafke} in Theorem \ref{dstiff} and hence $\NIN_{[0,1]}$.

\smallskip
  
Finally, \emph{a set has cardinality at most that of $2^{\N}$} can be expressed in various ways and we study the associated connections.  
We need the following definition. 
\bdefi[Sets of higher rank]
A set $X$ of type $\sigma$ objects is given by a characteristic function of type $\sigma \di \{0,1\}$.  
We assume such a set comes with an equivalence relation `$=_{X}$' that satisfies $x=_{\sigma}y \di x =_{X}y$ for all $x, y\in X$.  
\edefi
Regarding the previous definition, we observe that real equality satisfies $x=_{1}y\di x=_{\R}y$ for all $x, y\in \R$.    
As a further example, $X$ could be $C(2^{\N})$ and equality $F=_{X}G $ for $F^{2}, G^{2}\in C(2^{\N})$ is just $(\forall \sigma \in 2^{<\N})( F(\sigma*00\dots )=_{0}G(\sigma*00\dots)))$, which can be defined in $\ACAo$.  

\smallskip

By the following theorem, it seems that $\QFAC^{\sigma, 1}$ is necessary to go from `surjection from $2^{\N}$ to $X$' to `injection from $X$ to $2^{\N}$'.  
Note that surjections and injections are defined relative to $=_{X}$.  
\begin{thm}[$\ACAo+\QFAC^{\sigma, 1}$]
Let $X$ be a set of type $\sigma$-objects and $F:2^{\N}\di X$ be surjective, i.e.\ $(\forall x\in X)(\exists f\in  2^{\N})(F(f)=_{X}x)$.  
Then there is $G^{\sigma \di 1}$ which is an injection from $X$ to $2^{\N}$, i.e.\ $(\forall x, y\in X)(G(x)=_{1}G(y)\di x=_{X}y)$.  
\end{thm}
\begin{proof}
Let $F:2^{\N}\di X$ be surjective, i.e.\ $(\forall x^{\sigma}\in X)(\exists f\in 2^{\N})(F(f)=_{X}x)$.  
Now apply $\QFAC^{\sigma, 1}$ to obtain $\Phi^{\sigma\di 1}$ with $\Phi(x)=_{1}f$ in the previous.  
For $x,y\in X$, we have that $\Phi(x)=_{1} \Phi(y)$ implies $F(\Phi(x))=_{\sigma}F(\Phi(y))$ by extensionality, and hence $x=_{X}F(\Phi(x))=_{X}F(\Phi(y))=_{X}y$ by definition.  
\end{proof}
We observe that $\Phi$ from the proof of the theorem need not satisfy extensionality relative to $=_{X}$, i.e.\ we may have that $x=_{X}y$ (but $x\ne_{\sigma} y$) and $\Phi(x)\ne_{1}\Phi(y)$.

\smallskip

In conclusion, we have identified a couple of places in higher-order RM where $\QFAC^{\sigma,1}$ is used in a non-trivial and perhaps essential way.

\subsection{Streamlining known results}\label{fruitcake}
We discuss how quantifier-free choice as in $\QFAC^{\sigma, 1}$ can streamline known results.  
 
\smallskip
 
First of all, the RM of basic topology proceeds smoothly in the presence of $\QFAC^{1, 1}$ (\cite{sahotop}).  In particular, various rather different versions of the Heine-Borel theorem and Lindel\"of lemma are equivalent (see also \cite{dagsamXI}*{\S3.1.4}).  In a nutshell, we have formulated \emph{Cousin's lemma} from \cite{cousin1} as follows in e.g.\ \cite{dagsamIII}.  
\begin{princ}[$\HBU$]
For $\Psi:[0,1]\di \R^{+}$, there are $x_{0}, \dots, x_{k}\in [0,1]$ such that $\cup_{i\leq k}B(x_i, \Psi(x_{i}))$ covers $[0,1]$.
\end{princ}
The meaning of $\HBU$ is clear:  the uncountable covering $\cup_{x\in [0,1]}B(x, \Psi(x))$ has a finite sub-covering.    
Unfortunately, the notion of covering in $\HBU$ is too restricted for the development of topology.  In particular, the assumption that $x\in B(x, \Psi(x))$ is too strong. 
The following generalisation from \cite{sahotop} is equivalent assuming $\QFAC^{1,1}$.
\begin{princ}[$\HBT$]
For $\psi:[0,1]\di \R^{+}\cup\{0\}$ such that $(\forall x\in [0,1])(\exists y)(x\in B(y, \Psi(y)))$, there are $x_{0}, \dots, x_{k}\in [0,1]$ such that $\cup_{i\leq k}B(x_i, \psi(x_{i}))$ covers $[0,1]$.
\end{princ}

\smallskip

Secondly, we have investigated the following `explosive' combinations.  
\begin{itemize}
\item The combination of $\ACAo$ and the Lindel\"of lemma for the Baire space, proves $\FIVE$ (\cite{dagsamV});
\item The combination of $\ACAo+\QFAC^{2,1}$ and the Lindel\"of lemma for the Baire space, proves the Suslin functional as in $(\SS^{2})$ (\cite{dagsamV});
\item The combination of $(\SS^{2})$ with the Jordan decomposition theorem yields $\SIX$ (\cite{dagsamXI}). 
\end{itemize} 
Combing these items, $\SIX$ follows from of the combination of $\ACAo+\QFAC^{2,1}$, the Lindel\"of lemma for Baire space, and the Jordan decomposition theorem.

\smallskip

Thirdly, in the presence of $\QFAC^{1, 1}$, the functional $\SS_{k}^{2}$, which decides $\Sigma_{k}^{1}$-formulas, yields the Feferman-Sieg operators $\nu_{k}^{2}$ from \cite{boekskeopendoen}*{p.\ 129}, which returns a witness (if existent) to the $\Sigma_{k}^{1}$-formula at hand.  
Similarly, $\QFAC^{2,1}$ suffices to prove the equivalence between Kleene's quantifier $(\exists^{3})$ and Hilbert-Bernays' operator $\nu$ from \cite{hillebilly2}*{p.\ 479}.

\smallskip

Fourth, filters and nets can characterise the same (topological) properties and one can switch between them (\cite{zonderfilter}).   
A filter is however always of higher type than the associated net.   
To convert a (fourth-order) filter to a (third-order) net as done in \cite{zonderfilter}, $\QFAC^{2,1}$ seems to suffice.

\end{document}